\documentclass[10pt,a4paper,leqno]{article}
\usepackage[utf8]{inputenc}
\usepackage{amsmath}
\usepackage{graphicx}%
\usepackage{amsfonts}
\usepackage{amssymb}
\usepackage{comment}
\usepackage{hyperref}
\usepackage{enumitem,pstricks,xy,pst-node}
\xyoption{all}
\usepackage{amsthm}
\usepackage{longtable}
\usepackage{mathtools}
\usepackage{caption,subcaption} 
\usepackage{booktabs} 
\usepackage[color]{changebar} 
\cbcolor{red} 
\usepackage{mathrsfs} 
\usepackage{verbatim} 
\usepackage[disable] 
{todonotes}
\usepackage[normalem]{ulem} 
\usepackage{faktor} 
\usepackage{caption} 
\usepackage{nicefrac} 
\usepackage{extarrows} 
\usepackage{enumitem} 
\newtheorem{theorem}{Theorem}[section]

\newtheorem{proposition}[theorem]{Proposition}

\newtheorem{definition}[theorem]{Definition}
\theoremstyle{definition}

\theoremstyle{remark}

\newcommand{\sfrac}[2]{{#1}/{#2}}

\newcommand{\PP}{\mathbb{P}}

\DeclareMathOperator{\proj}{Proj}
\DeclareMathOperator{\Hb}{H}
\DeclareMathOperator{\sln}{SL}

\DeclareMathOperator{\s}{\mathcal{S}}

\newcommand{\at}{\operatorname{Aut}}

\newcommand{\ladi}{\begin{lastadd}}
\newcommand{\ladf}{\end{lastadd}}
\newcommand{\lrei}{\begin{lastrem}}
\newcommand{\lref}{\end{lastrem}}
\newenvironment{lastadd}
{\cbstart\color{red}}
{\todo{red to remove}\cbend}
\newenvironment{lastrem}
{\cbstart\color{yellow}}
{\cbend}

\newcommand{\bigslant}[2]{{\raisebox{.2em}{$#1$}\left/\raisebox{-.2em}{$#2$}\right.}}

\author{Hanieh Keneshlou
\thanks{Max Planck Institute for Mathematics in Sciences, 04103 Leipzig, Germany. \texttt{hanieh.keneshlou@mis.mpg.de}}
}
\title{Cubic surfaces on the singular locus of the Eckardt hypersurface}
\begin{document}
\maketitle
\begin{abstract}
The Eckardt hypersurface in $\PP^{19}$ parameterizes smooth cubic surfaces with an Eckardt point, which is a point common to three of the $27$ lines on a smooth cubic surface. We describe the cubic surfaces lying on the singular locus of the model of this hypersurface in $\PP^4$, obtained via restriction to the space of cubic surfaces possessing a so-called Sylvester form. We prove that 
inside the moduli of cubics, the singular locus corresponds to a reducible surface with two rational irreducible components intersecting along two rational curves. The two curves intersect in two points corresponding to the Clebsch and the Fermat cubic surfaces. We observe that the cubic surfaces parameterized by the two components or the two rational curves are distinguished by the number of Eckardt points and automorphism groups.
\end{abstract}
\section*{Introduction}
The moduli of cubic surfaces $\mathcal{M}_{cub}$ is defined as the geometric invariant quotient of the space of quaternary cubics $\PP^{19}\cong \PP(\Hb^0(\PP^3, \mathcal{O}_{\PP^3}(3)))$ by the induced action of the special linear group $\sln(4)$. The classical
description of this space is due to G. Salmon \cite{sl} and A. Clebsch \cite{cl}. They proved $\mathcal{M}_{cub}$ is isomorphic to the weighted projective space $\PP(1,2,3,4,5)$ by showing the existence of six homogeneous invariant polynomials $\bar{I}_n$ of degrees $n=8,16,24,32, 40, 100$, which generate the corresponding graded ring of invariants. The first five polynomials are algebraically independent, however $\bar{I}_{100}^2$ can be expressed as a polynomial in terms of the other invariants. In fact, they described a birational model of $\mathcal{M}_{cub}$ as the quotient space of $\PP^4$, the parameter space of cubic surfaces with a so-called Sylvester form, by the action of the Symmetric group $\s_5$. Under this birational equivalence, each $\sln(4)$-invariant polynomial $\bar{I}_n$ can be regarded as continuation of a $\s_5$-invariant polynomial $I_{n}$ of the same degree in the coordinate ring of $\PP^4$. In this way, each of the coordinate hypersurfaces in the moduli $\PP(1,2,3,4,5)$ has a birational model in the quotient space of $\PP^4$ given by the Salmon $\s_5$-invariant.

One can use these invariants to describe further interesting subspaces of the space of cubic surfaces.  By classical results, the vanishing of $\bar{I}_{32}$ is a necessary and sufficient condition for a cubic surface to be singular. In this line, the famous discriminant hypersurface $V(\bar{I}_{32})$ parameterizes singular cubic surfaces generically having a node. Turning to smooth cubic surfaces equipped with 27 lines, a point common to three lines is called an Eckardt point. The Salmon invariant $\bar{I}_{100}$ vanishes on the closure of the locus of smooth cubic surfaces with an Eckardt point. The hypersurface $V(\bar{I}_{100})\subset \PP^{19}$ is called the Eckardt hypersurface.

Let $E= V(I_{100})$ be the model of the Eckardt hypersurface in $\PP^4$, parameterizig the cubic surfaces in Sylvester form having an Ecklardt point. By abuse of language, we may refer to $E$ as Eckardt hypersuarfce as well. The main contribution of this paper is to study the cubic surfaces determined by the singular locus of this hypersurface. We prove (Theorem \ref{my}), up to linear change of coordinates in $\PP^3$, the singular locus determines two 2-dimensional rational families of cubic surfaces intersecting along two rational curves. The two curves intersect at two points which corespond to the Clebsch and the Fermat cubic surfaces possessing respectively $10$ and $18$ Eckardt points. The generic elements of the two families are smooth cubic surfaces with respectively $2$ and $3$ Eckardt points. Moreover, the two rational curves parameterize the cubic surfaces with respectively $4$ and $6$ Eckardt points. The difference in number of the Eckardt points implies then further difference in the automorphism group of the cubic surfaces parameterized in different families.

The paper is structured as follows. In the first section, we recall the construction of the moduli of cubic surfaces as a weighted projective space. Section \ref{locus} deals with the general description of the singular locus of the Eckardt hypersurface $E$ corresponding to a reducible surface with two components inside the moduli of cubics. In Section \ref{des}, we investigate the geometric feature and the differences of the cubic surfaces lying on the two components and the two rational curves. 

Our results rely on the computations done by the computer algebra system \textit{Macaulay2} \cite{m2}, and using the supporting functions in \cite{hani}.
\section*{Acknowledgements}
I would like to thank Bernd Sturmfels for drawing my attention to the world of cubic surfaces. The present work addresses the question $13$ from \cite{bs}. I thank Avinash Kulkarni for his valuable comments on my manuscript.
\section{Preliminaries}\label{pre}
In this section we briefly review the construction of the moduli spaces of cubic surfaces using the Sylvester forms and the Salmon's invariants.\\

Let $\mathbb{K}$ be a field and $\PP^{19}=\PP(V)$ be the parameter space of cubic forms $V=\mathbb{K}[x_0,\ldots,x_3]_{(3)}$ in four variables. One considers the induced action of $G:=\sln(4)$ on $\PP^{19}$, from the standard action on $\PP^{3}$. The geometric invariant quotient
\vskip -0.1cm
$$\mathcal{M}_{cub}:=\bigslant{\PP^{19}}{G}$$
is called the moduli space of cubic surfaces. By the following result, $\mathcal{M}_{cub}$ is a projective variety.
\begin{proposition}
There is an isomorphism $\mathcal{M}_{cub}\cong \proj(R^G)$, where $R$ is the coordinate ring of $\PP^{19}$ and $R^G$ is the ring of invariants.
\end{proposition}
\begin{proof}
See \cite{dol3}, Proposition 8.1.
\end{proof}
 The computation from classical invariant theory due to \cite{sl} and \cite{cl} indicates that the graded ring of invariants is generated by homogeneous polynomials $\bar{I}_n$ of degrees
$$  n=8,16,24,32,40,100 $$
such that the first five invariants are algebraically independent. There is a relation expressing $\bar{I}_{100}^2$ as a polynomial in terms of the remaining invariants. Therefore, the graded subalgebra generated by
elements of degree divisible by 8 is freely generated by the first five invariants, and $\mathcal{M}_{cub}$ has structure of the weighted projective space 
$$\mathcal{M}_{cub}\cong \mathbb{P}(1,2,3,4,5).$$
One can restrict the invariants to an open subset of cubic surfaces with somewhat easier form, that is the set of the cubic surfaces with Sylvester forms. This then would allow to express the invariants in terms of symmetric functions of the coefficients of the Sylvester representation.
\begin{theorem}A general cubic surface is projectively isomorphic to a surface in $\PP^4$ given by equations
$$a_0z_0^3+a_1z_1^3+a_2z_2^3+a_3z_3^3+a_4z_4^3=0,\quad z_0+z_1+z_2+z_3+z_4=0.\quad (\ast)$$
The coefficients $a_0,\ldots,a_4$ are determined uniquely up to permutation and a common scaling.
\end{theorem}
\begin{proof}
See \cite{dolbook}, Corollary 9.4.2.
\end{proof}
\begin{definition}
A cubic surface given by equations as $(\ast) $ is said to have a Sylvester form. It has non-degenerate Sylvester form if $a_i\neq 0$ for all $i=1,\ldots,4$. Otherwise, it has a degenerate Sylvester form.
\end{definition}
Salmon's computations then provide the following easy formulation of the invariants for cubic surfaces possessing a Sylvester representation:
$$I_8=\sigma_4^2-4\sigma_3\sigma_5,\ I_{16}=\sigma_1\sigma_5^3,\ I_{24}=\sigma_4\sigma_5^4,\ I_{32}=\sigma_2\sigma_5^6,\ I_{32}=\sigma_5^8, $$
and
$$ I_{100}=\sigma_5^{18}.\det \begin{pmatrix}
1 & a_0& a_0^2 & a_0^3 & a_0^4 \\
 \vdots & & \ddots& & \vdots\\
1 & a_4& a_4^2 & a_4^3 & a_4^4 \\
\end{pmatrix}$$
where $\sigma_i$ is the elementary symmetric polynomial of degree $i$ in $a_0,\ldots,a_4$. One observes that in this formulation, the $G$-invariant polynomials can be viewed as invariants under the action of the symmetric group $\s_5$. Furthermore, let $\sfrac{\PP^{4}}{\s_5}$ be the quotient of the parameter space $\PP^4$ of Sylvester forms by the action of the symmetric group $\s_5$. This quotient space is isomorphic to the weighted projective space $\PP(1,2,3,4,5) $ equipped with natural coordinates $\sigma_1,\ldots,\sigma_5$. In shadow of the above facts, there is a birational map
\begin{small}
\begin{displaymath}
\xymatrix{
\bigslant{\PP^{4}}{\s_5}\cong \PP(1,2,3,4,5)  \ar@{-->}[r] & \mathcal{M}_{cub}\cong \PP(1,2,3,4,5)
}
\end{displaymath}
\vspace{-0.5cm}
$$\hspace{0.2cm} (\sigma_1:\sigma_2:\sigma_3:\sigma_4:\sigma_5)\longmapsto(I_8:I_{16}:I_{24}:I_{32}:I_{40}).$$
\end{small}
\noindent
\hspace{-0.15cm}with base locus $V(\sigma_4,\sigma_5)$. The birational inverse is defined by
\begin{small}
$$(I_8:I_{16}:I_{24}:I_{32}:I_{40})\longmapsto \left(\dfrac{I_{16}}{\sigma_5^3}:\dfrac{I_{32}}{\sigma_5^6}:\dfrac{I_{24}I_{40}}{\sigma_5^{12}}:\dfrac{I_{24}^2-I_8I_{40}}{4\sigma_5^9}:\dfrac{I_{40}^2}{\sigma_5^{15}}\right), $$
\end{small}
\noindent
\hspace{-0.1cm}and has base locus at the point $Q=(1:0:0:0:0)$. In this way, each of the coordinate hypersurfaces in $\mathcal{M}_{cub}\cong\PP(1,2,3,4,5)$ has a birational model in the quotient space of $\PP^4$ defined by the Salmon $\s_5$-invariant.
\section{Singular locus of the Eckardt hypersurface}\label{locus}
A point where three lines in a smooth cubic surface intersect is called an Eckardt point. In this case, the three lines are cut out by the intersection of the cubic surface with the tangent plane at this point. The Eckardt hypersurface $V(\bar{I}_{100})\subset \PP^{19}$ parameterizes the smooth cubic surfaces with an Eckardt point. 
Let $ E:V(I_{100})\subset \mathbb{P}^4$ be the model of the Eckardt hypersurface inside the space of cubic surfaces possessing a Sylvester form.  The following theorem describes a general cubic surface (up to linear change of coordinates in $\PP^3$) lying on the singular locus $\Gamma\subset E$ of this hypersurface. More precisely, let 
\begin{align*}
 \Phi:\mathbb{P}^4 &\longrightarrow \mathcal{M}_{cub} \cong\mathbb{P}(1,2,3,4,5)\\
\quad (a_0:\cdots:a_4)& \mapsto (I_8:I_{16}:I_{24}:I_{32}:I_{40}). 
\end{align*}
be the dominant map obtained via the quotient map from $\PP^4$, and set $\Delta:=\Phi(\Gamma)$, then we have:
\begin{theorem}\label{my}
With the above notation, $\Delta$ is a reducible surface with two equidimensional rational irreducible components $S_1$ and $S_2$. The generic point of each component corresponds to a smooth cubic surface with the Sylvester form as follows, respectively:
$$S_1:\quad az_0^3+bz_1^3+bz_2^3+cz_3^3+cz_4^3=0,\quad z_0+z_1+z_2+z_3+z_4=0,\quad a,b,c\in \mathbb{C} $$
$$S_2:\quad az_0^3+bz_1^3+bz_2^3+bz_3^3+cz_4^3=0,\quad z_0+z_1+z_2+z_3+z_4=0,\quad a,b,c\in \mathbb{C}.$$
The two components intersect along two rational curves $C_1$ and $C_2$. The two curves intersect at two points which represent the Clebsch and the Fermat cubic surfaces.
\end{theorem}
\begin{proof}
An explicit computation (see \texttt{VerifyAssertion1}, \cite{hani}) of the singular locus $\Gamma$ and its decomposition to irreducible components indicates that $\Gamma$ has 30 irreducible linear components as follows:
\begin{itemize}
    \item 5 irreducible components corresponding to the coordinate hyperplanes $$V(a_i)\subset \PP^4,\quad i=0,\ldots,4.$$
\end{itemize}    
    Therefore, under the map $\Phi$, the five components parameterizing the cubic surfaces with a degenerate Sylvester form are contracted to the point $Q$, which corresponds to the Fermat cubic surface
 $$S_F:\quad (x_0^3+x_1^3+x_2^3+x_3^3=0)\subset \PP^3.$$   
    \begin{itemize}
        \item 15 irreducible components as copies of $\mathbb{P}^2$ given by
    $$V_{ijkl}=V(a_i-a_j, a_k-a_l)\subset \PP^4, \quad i\neq j\neq k\neq l$$
  for $(i,j,k,l)$ among 
    \begin{align*}
         &(2,3,1,4),\ (1,4,0,2),\ (1,4,0,3),\ (2,4,1,3),\ (1,3,0,2)\\
         &(1,3,0,4),\ (3,4,1,2),\ (3,4,0,1),\ (3,4,0,2),\ (2,3,0,1)\\
          &(2,3,0,4),\ (1,2,0,3),\ (1,2,0,4),\ (2,4,0,1),\ (2,4,0,3)
    \end{align*}   
    \end{itemize}
Since for a pair of 4-tuples $(i,j,k,l), (i',j',k',l')$, one can map the corresponding components $V_{ijkl}, V_{i'j'k'l'}$, one to another, by a permutation of the coordinates in $\PP^4$ (and hence keeping the hyperplane $V(\sum_{i=0}^4 a_i)$ invariant), the two components, and therefore all the components $V_{ijkl}$'s are mapped to a rational surface $S_1\subset \mathcal{M}_{cub}$ whose general point is a cubic surface with Sylvester form:
$$S_1:\quad az_0^3+bz_1^3+bz_2^3+cz_3^3+cz_4^3=0,\quad z_0+z_1+z_2+z_3+z_4=0,\quad a,b,c\in \mathbb{C}.$$
Furthermore, choosing the coordinate $(a:b:c)$ for $\PP^2$, as the representative of the components $V_{ijkl}$, we have
$$\PP^2\setminus V(abc) \longrightarrow  S_1\setminus \lbrace Q\rbrace.$$
\begin{itemize}
        \item 10 irreducible components as copies of $\mathbb{P}^2$ given by
    $$V_{ijk}=V(a_i-a_j,\ a_k-a_j)\subset \PP^4, \quad i\neq j\neq k$$
    for $(i,j,k)$ among the triples
    \begin{align*}
         &(2,4,1),\ (3,4,2),\ (3,4,1),\ (3,4,0),\ (1,4,0)\\
         &(2,3,0),\ (2,3,1),\ (1,2,0),\ (1,3,0),\ (2,4,0)
    \end{align*}
\end{itemize}
With the same argument as above, for two choices of $(i,j,k)$, the points on the corresponding components differ only by a coordinate permutation of $\PP^4$, and hence are mapped to a rational surface $S_2\subset \mathcal{M}_{cub}$ whose general point is a cubic surface of type:
$$S_2:\quad az_0^3+bz_1^3+bz_2^3+bz_3^3+cz_4^3=0,\quad z_0+z_1+z_2+z_3+z_4=0,\quad a,b,c\in \mathbb{C}.$$
It is clear that the two surfaces intersect along two rational curves $C_1$ and  $C_2$, as the images of the two lines $\ell_1:\PP^1\cong  V(a-c)\subset\PP^2$ and $\ell_2:\PP^1\cong  V(b-c)\subset\PP^2$, respectively. Therefore, they parameterize the cubic surfaces with the Sylvester forms as follows:
$$C_1:\quad az_0^3+bz_1^3+bz_2^3+bz_3^3+az_4^3=0,\quad z_0+z_1+z_2+z_3+z_4=0,\quad a,b\in \mathbb{C}.$$
$$C_2:\quad az_0^3+bz_1^3+bz_2^3+bz_3^3+bz_4^3=0,\quad z_0+z_1+z_2+z_3+z_4=0,\quad a,b\in \mathbb{C} $$
In particular, the two curves intersect at two points, $Q$ and the point which represents the Clebsch cubic surface given by the following equations:
$$S_c:\quad z_0^3+z_1^3+z_2^3+z_3^3+z_4^3=0,\quad z_0+z_1+z_2+z_3+z_4=0.$$
\end{proof}
\vspace{-1cm}
\section{How different are the two components?}\label{des}
In this section, we study the geometric feature and the difference of the cubic surfaces lying on the two components or the two curves. In this direction, an explicit examination demonstrates that the two components stand for two different types of singularities of $E$. In fact, identifying each of the components with its possible birational model as an irreducible component of $\Gamma$ we have:
\begin{theorem}
A general point of $S_1$ (resp. $S_2$) corresponds to an ordinary double (resp. triple) point on $E$. Moreover, the generic point of each of the two rational curves corresponds to an ordinary triple point on $E$.
\end{theorem}
\begin{proof}
See (\texttt{VerifyAssertion2}, \cite{hani}) for the explicit computation.
\end{proof}
The nicely prescribed Sylvester forms of the cubic surfaces on the two components and the curves reveal the number of Eckardt points and their arrangements. More precisely, for a generic cubic surface, let $ L:=L_i,\ i=0,\ldots,4$ be the five linear forms in its Sylvester representation such that any four of them are linearly independent. To this data, one can associate the Sylvester pentahedron with faces
$$\pi:\pi_i:=(L_i=\ell=0),\ \text{where}\ \ell:=\sum_{i=0}^{4}L_i,$$
such that two faces $\pi_i,\pi_j$ intersect along the edge $\beta_{ij}:=(L_i=L_j=\ell=0)$, and the vertex $A_{ij}$ is the intersection point of the three faces with indices different from $i,j$. By classical results (\cite{segre}, page 148), one can see that in this notation, a general cubic surface $S$ on
\begin{itemize}
    \item $S_1$ has 2 Eckardt points $A_{12},A_{34}$ such that the joining line $(z_1+z_2=z_3+z_4=\sum z_i=0)$ is contained in $S$.
    \item $S_2$ has 3 Eckardt points given by the vertices $A_{12}, A_{23}, A_{13}$, which are collinear and the common line is not contained in $S$.
     \item $C_1$ has 4 Eckardt points $A_{12}, A_{23}, A_{13}$ which are collinear and the point $A_{04}$, which is joined to the former points by the three lines through it.
     \item $C_2$ has 6 Eckardt points. The 6 points are the 6 vertices of the quadrilateral intersected on  $\pi_0$ by four other faces of the pentahedron.
\end{itemize}
In particular, the Clebsch cubic surface has $10$ Eckardt points as the $10$ vertices of the Pentahedron. The Fermat cubic surface possesses $18$ Eckardt points.\\

The diversity in number of Eckardt points causes further difference in terms of the automorphism group. An automorphism of the projective space $\PP^3$ fixing a hyperplane and a point is called a homology. The single point is called the center of homology. In terminology of classical projective geometry, a homology of order 2 is usually referred to as an involution. There is a one-to-one correspondence between the set of Eckardt points of a smooth cubic surface and the set of involutions of $\PP^3$ keeping the surface invariant (\cite{dd}, Theorem 9.2). 

To avoid repeating, let $S$ denote the cubic surface corresponding to a general point of one of the two surfaces or the two curves. By classification of the possible groups of automorphisms of a smooth cubic surface, a general cubic lying on $S_1$ has automophism group generated by the two involutions associated to the two Eckardt points and $\at(S)\cong (\mathbb{Z}_2)^2.$
On the other hand, the automorphism group of a general cubic surface lying on $S_2$ is generated by involutions permuting the three Eckardt points and keeping the common line invariant, that is $\at(S)=\s_3$. For a general cubic surface on the curves $C_1$ and $C_2$, one has $\at(S)\cong \s_3\times \s_2$ and $\at(S)\cong \s_4$, respectively. The Clebsch cubic surface has the automorphism group $\at(S_c)\cong \s_5$ acting by permutations of the coordinates in $\PP^4$. Up to isomorphism, the Clebsch surface is the only cubic surface with this automorphism group. The automorphism group of the Fermat cubic surface is isomorphic to $H \rtimes\s_4$
where the subgroup $H$ acts by multiplying the coordinates by a primitive third root
of unity and $\s_4$ acts by permuting the coordinates in $\PP^3$. For more on automorphism group of cubic surfaces classified in any characteristic, we refer the reader to the recent paper \cite{dd}.

\end{document}